\documentclass[11pt]{elsarticle}

\UseRawInputEncoding

\usepackage{amssymb}
\usepackage{amsmath}
\usepackage{amsthm}

\usepackage[colorlinks]{hyperref}
\AtBeginDocument{%
  \hypersetup{%
    linkcolor=blue,%
    citecolor=green,%
  }%
}
\usepackage{cleveref}
\usepackage{geometry}
\crefname{equation}{}{}

\usepackage{graphicx}
\usepackage{setspace}
\usepackage{paralist}
\usepackage{longtable}
\usepackage{multirow}
\usepackage{rotating}
\usepackage{marginnote}
\usepackage{mathrsfs}
\usepackage{float}
\usepackage{hyperref}
\usepackage{comment}
\usepackage{color}
\usepackage{indentfirst}
\usepackage{footmisc}

\newtheorem{theorem}{Theorem}[section]
\newtheorem*{theorem*}{Theorem}
\newtheorem{lemma}[theorem]{Lemma}

\theoremstyle{definition}
\newtheorem{definition}[theorem]{Definition}

\theoremstyle{remark}
\newtheorem{remark}[theorem]{Remark}

\numberwithin{equation}{section}

\begin{document}

\begin{frontmatter}
\title{Uniqueness and some related estimates for Dirichlet problem with fractional Laplacian\tnoteref{t1}}

\author[rvt]{Congming Li}
\ead{congming.li@sjtu.edu.cn}
\author[rvt]{Chenkai Liu}
\ead{Lck0427@sjtu.edu.cn}

\tnotetext[t1]{Congming Li \& Chenkai Liu are partially supported by NSFC12031012, NSFC11831003 and  Institute of Modern Analysis-A Frontier Research Center of Shanghai.}

\address[rvt]{School of Mathematical Sciences, Shanghai Jiao Tong University, China;}

\begin{abstract}
For the fractional Laplace equation, a surprising observation is the non-uniqueness for the basic Dirichlet type problems.
In this paper, a somewhat sharp uniqueness condition for the fractional Laplace equation is established.  We derive the $L^p$-estimate for  fractional Laplacian operators to better understand this phenomena. Several weighted fractional Sobolev spaces appear naturally.  We then establish the embedding relations between these spaces. These existence-uniqueness conditions and the spaces we introduce here are intrinsically related to the fractional Laplacian. These are basic properties to the fractional Laplace equations and can be useful in the study of related problems.
\end{abstract}


\end{frontmatter}

%\footnotetext{* Corresponding authors: lck0427@sjtu.edu.cn}

\section{Introduction}

The $W^{2,p}$ problem plays an essential role in the classical elliptic PDE theory
: Consider the following problem \cite{Gilbarg1991Elliptic}:
\begin{equation}\label{pb0}
\left\{
\begin{aligned}
-\Delta u+\vec b\cdot\nabla u+cu&=f\ \ \ &\text{in}\ &\Omega,\\
u&=0\ \ \ &\text{on}\ &\partial \Omega.
\end{aligned}
\right.
\end{equation}
Assuming $\Omega\subset \mathbb{R}^n$ be a bounded domain with smooth boundary,  $f\in L^p(B_1)$, $\vec b, c\in L^\infty(\Omega)$ and $c\geqslant0$ in $\Omega$, one derives the existence and uniqueness of (strong) solution
$u\in W^{2,p}(\Omega)\cap W^{1,p}_0(\Omega)$ of \eqref{pb0}. Moreover, the solution $u$ admits the following apriori estimate:
\begin{equation}
\|u\|_{W^{2,p}(\Omega)}\leqslant C\|f\|_{L^p(\Omega)}.
\end{equation}

The Dirichlet type problems of fractional Laplacian arose for its applications in physical sciences, probability and finance.
Many interesting results have emerged in recent years (see \cite{1996Local,2006Classification,2010Regularity,2015Non,2016Existence,D2016ON,2017Maximum,2017On,2019Symmetry}). In this paper, we first consider the following problem:
\begin{equation}\label{pb1}
\left\{
\begin{aligned}
(-\Delta)^s u&=f\quad &\text{in}\ &\mathcal{D}'(B_1),\\
u&=0\quad&\text{in}\ &\mathbb{R}^n\backslash B_1.
\end{aligned}
\right.
\end{equation}

Very different from case of  the classical Laplacian, for the problem \eqref{pb1} with fractional Laplace, one can construct a non-trivial solution for $f\equiv 0$:
\begin{equation}
\label{nontvl}
u(x)= \begin{cases}
\displaystyle c(n,s)\int_{\partial B_1}\frac{(1-|x|^2)^s}{|x-y|^n}\ d\mathcal{H}^{n-1}_y=C(n,s)(1-|x|^2)^{s-1}, & \text{for}\  x\in B_1;\\
0, & \text{for}\  x\in \mathbb{R}^n\backslash B_1.
\end{cases}
\end{equation}

We search for a sharp condition for the solution to be unique.
It  is introduced in \cite{silvestre} that maximum principle holds for \eqref{nontvl} if lower semi-continuity is assumed.
As a result, the continuous solution $u$ of \eqref{pb1} must be unique. However,  the continuous assumption is too strong.
Inspired by the work of \cite{abatangelo2018green} and \cite{ros}, we come up with an idea that give a trace-type existence and uniqueness condition for problem \eqref{pb1} as follow:

For convenience, we denote
$\delta=\operatorname{dist}(\cdot, \mathbb{R}^n\backslash \Omega)$ and define the weighted $L^p$-space to be:
\begin{equation}
L^p_r(\Omega)=\{f\in L^1_{\mathrm{loc}}(\Omega)|\|f\|_{L^p_r(\Omega)}:=\|\delta^rf\|_{L^p(\Omega)}<\infty\},
\end{equation}
\begin{theorem}
\label{cor1}
  Let $0<s<1$, then problem  \eqref{pb1} has at most one solution $u\in \mathcal{L}_{2s}$ that satisfies:
\begin{equation}
\label{cd:un1}
  \lim_{\epsilon\rightarrow 0}\frac{1}{\epsilon^s}\int_{\{x\in B_1|\delta(x)\leqslant\epsilon\}}|u(x)| d x=0.
\end{equation}
Moreover, if  $f\in L^1_s(B_1)$,  then the solution exists.
\end{theorem}
\begin{remark}
Indeed, for the regular Laplacian case:
\begin{equation}
\left\{
\begin{aligned}
-\Delta u&=0\ \ \ &\text{in}\ &\Omega,\\
u&=0\ \ \ &\text{on}\ &\partial \Omega.
\end{aligned}
\right.
\end{equation}
The uniqueness condition $u=0$ can be understood in the following way:
\begin{equation}
  \lim_{\epsilon\rightarrow 0}\frac{1}{\epsilon}\int_{\{x\in B_1|\delta(x)\leqslant\epsilon\}}|u(x)| d x=0.
\end{equation}

In this sense,  Theorem \ref{cor1} can be seen as a fractional case generalization.
\end{remark}
\begin{remark}
\label{rmk:nonun}
The condition \eqref{cd:un1} in Theorem \ref{cor1} is sharp for the uniqueness to hold. Indeed, one can consider the nontrivial solution \eqref{nontvl}.
It is easily shown that $u$ satisfies:
\begin{equation}
\label{cd:un0}
  \lim_{\epsilon\rightarrow 0}\frac{1}{\epsilon^s}\int_{\{x\in B_1|\delta(x)\leqslant\epsilon\}}|u(x)| d x=c>0.
\end{equation}
\end{remark}

In the following, we consider a more complicated  problem(see applications in \cite{cabre1,caffarelli2,chen2,li0}):
\begin{equation}\label{pb3}
\left\{
\begin{aligned}
(-\Delta)^s u+\vec b\cdot\nabla u+cu&=f\quad &\text{in}\ &\mathcal{D}'(B_1),\\
u&=0\quad&\text{in}\ &\mathbb{R}^n\backslash B_1.
\end{aligned}
\right.
\end{equation}
Here, in order to ensure the ellipticity of this equation, we need $\vec b\equiv 0$ if $2s\leqslant1$.

In our study, we introduce several subtle estimates and apply the method of continuity to get the following solvability theorem of Dirichlet problem \eqref{pb3}.
\begin{theorem}
\label{exi}
Let $\frac{1}{2}<s<1$, $1-s\leqslant r\leqslant s$, $1\leqslant p<\infty$, suppose $\vec b, c\in L^\infty(B_1)$, $c\geqslant 0$ in $B_1$ and $f\in  L^p_r(B_1)$. Then  there exists a unique solution $u\in\mathcal{L}_{2s}$
 of \eqref{pb3} that satisfies the  condition \eqref{cd:un1}. Furthermore, we have the apriori estimate: for a constant $C>0$ depending only on $n,s,p$ and $r$.
\begin{equation}
  \|(-\Delta)^su\|_{L^{p}_{r}(B_1)} \leqslant C \|f\|_{L^p_r(B_1)}.
\end{equation}
The derivative of $u$ can also be estimated:

\begin{enumerate}
\item If $p=1$, then for any $q$ satisfying $1\leqslant q<\frac{n}{n-2s+1}$,
$|\nabla u|\in  L^q_r(B_1)$ with
\begin{equation}
\||\nabla u|\|_{ L^q_r(B_1)}\leqslant C\|f\|_{ L^1_r(B_1)}.
\end{equation}
\item If $1<p<\frac{n}{2s-1}$, then $|\nabla u|\in  L^{\frac{np}{n-(2s-1)p}}_r(B_1)$ with
\begin{equation}
\||\nabla u|\|_{ L^{\frac{np}{n-(2s-1)p}}_r(B_1)}\leqslant C\|f\|_{ L^p_r(B_1)}.
\end{equation}
\item If $p>\frac{n}{2s-1}$, then $|\nabla u|\in  L^\infty_r(B_1)$ with
\begin{equation}
\||\nabla u|\|_{ L^\infty_r(B_1)}\leqslant C\|f\|_{ L^p_r(B_1)}.
\end{equation}
\end{enumerate}

\end{theorem}

\section{Preliminaries}
\label{2}

For $0<s<1$ and $u\in C_0^\infty(\mathbb{R}^n)$, the fractional Laplacian $(-\Delta)^su$ is given by
\begin{equation}\label{1.0}
(-\Delta)^su(x)=\mathcal{F}^{-1}[(2\pi|\xi|)^{2s}\mathcal{F}[u](\xi)](x)=C_{n,s}{\rm P.V.}\int_{\mathbb{R}^n}\frac{u(x)-u(y)}{|x-y|^{n+2s}}dy,
\end{equation}
Here P.V. stands for the Cauchy principle value, $\mathcal{F}$ denotes the Fourier transform.

One can show that for the above type functions $u$, it holds that:
\begin{equation}
|(-\Delta)^{s} u(x)|\leqslant \frac{C}{1+|x|^{n+2s}}.
\end{equation}

In order that $(-\Delta)^sw$ make sense as a distribution, it is common to define:
\begin{equation*}
\mathcal{L}_{2s}=\bigg\{u: \mathbb{R}^n\rightarrow\mathbb{R}\bigg|\|u\|_{\mathcal{L}_{2s}}:= \int_{\mathbb{R}^n}\frac{|u(y)|}{1+|y|^{n+2s}}dy<+\infty\bigg\}.
\end{equation*}
Then for $w\in \mathcal{L}_{2s}$, $(-\Delta)^sw$  as a distribution is well-defined: $\forall \varphi\in C_0^\infty(\mathbb{R}^n)$,
\begin{equation}
\label{eq3}
(-\Delta)^sw[\varphi]=\int_{\mathbb{R}^n}w(x)(-\Delta)^s\varphi(x) dx.
\end{equation}
Indeed, the $\mathcal{L}_{2s}$ assumption just makes the integral on the right hand side of \eqref{eq3} converges.

For $w\in\mathcal{L}_{2s}$, $(-\Delta)^s w$ is naturally a distribution on $\Omega$, for $f\in L^1_{\rm loc}(\Omega)$ we say
\begin{equation}\label{eq2}
  (-\Delta)^s w=f\qquad\text{in}\quad\mathcal{D}'(\Omega),
\end{equation}
if for any test function $\varphi\in C_0^\infty(\Omega)$, it holds that
\begin{equation}
  (-\Delta)^sw[\varphi]=\int_{\mathbb{R}^n}w(x)(-\Delta)^s\varphi(x) dx=\int_{\Omega}f(x)\varphi(x) dx.
\end{equation}

On the other hand, for $w\in\mathcal{L}_{2s}\cap C^{1,1}_{\rm loc}(\Omega)$, $(-\Delta)^s w$ is pointwisely well-defined on $\Omega$ by the formula
\begin{equation}
(-\Delta)^su(x)=C_{n,s}{\rm P.V.}\int_{\mathbb{R}^n}\frac{u(x)-u(y)}{|x-y|^{n+2s}}dy\qquad\text{for}\quad x\in\Omega.
\end{equation}
Moreover, for $w\in\mathcal{L}_{2s}\cap C^{1,1}_{\rm loc}(\Omega)$, \eqref{eq2} indicates that
\begin{equation}
(-\Delta)^su(x)=f(x)\qquad\text{pointwisely\ for}\ x\in\Omega.
\end{equation}

Define the Poisson kernel $P(\cdot,\cdot)$ and Green's function $G(\cdot,\cdot)$ for $s$-Laplacian in $B_1$ \cite{bucur}  as:
\begin{equation}\label{P}
P(x,y)=c(n,s)\left(\frac{1-|x|^2}{|y|^2-1} \right)^s\frac{1}{|x-y|^n}, \ \ \text{for}\ x\in B_1, y\in\mathbb{R}^n\backslash B_1.
\end{equation}
\begin{equation}\label{G}
G(x,y)=\frac{\kappa(n,s)}{|x-y|^{n-2s}}\int_0^{\rho(x,y)}\frac{t^{s-1}}{(1+t)^\frac{n}{2}}\ dt, \ \ \text{for}\ x,y\in B_1,
\end{equation}
where $\rho(x,y)$ is defined as
\begin{equation}
\rho(x,y)=\frac{(1-|x|^2)(1-|y|^2)}{|x-y|^2}.
\end{equation}

Let $P\ast g$ and $G\ast f$ be defined as
\begin{equation}
  \label{Pg}
  P\ast g(x)=\begin{cases}
\displaystyle \int_{\mathbb{R}^n\backslash B_1} P(x,y)g(y)\ dy&\ \ \ \text{in}\  B_1,\\
g(x)&\ \ \ \text{in}\ \mathbb{R}^n\backslash B_1,
\end{cases}
\end{equation}

\begin{equation}\label{Gf}
G\ast f(x)=\begin{cases}
\displaystyle \int_{B_1} f(y)G(x,y)\ dy&\ \ \ \text{in}\  B_1,\\
0&\ \ \ \text{in}\ \mathbb{R}^n\backslash B_1.
\end{cases}
\end{equation}

The following two results are introduced in \cite{bucur}:
\begin{lemma}\label{exi0}
Let $0<s<1$, $f\in C^{2s+\epsilon}(B_1)\cap C(\overline{B_1})$. Then there exists a unique continuous solution: $u=G\ast f$ that solves the problem \eqref{pb1} pointwisely.
\end{lemma}
\begin{lemma}\label{exi1}
Let $0<s<1$, $g\in \mathcal{L}_{2s}\cap C(\mathbb{R}^n)$. Then there exists a unique continuous solution: $u=P\ast g$ that solves the following problem \eqref{pb2} pointwisely.
\begin{equation}\label{pb2}
\left\{
\begin{aligned}
(-\Delta)^s u&=0\quad &\text{in}\ &\mathcal{D}'(B_1),\\
u&=g\quad&\text{in}\ &\mathbb{R}^n\backslash B_1.
\end{aligned}
\right.
\end{equation}
\end{lemma}

One easily observe from the formula  of Green's function that:
\begin{lemma}\label{est3}
\begin{equation}
G(x,y)\leqslant \frac{C}{|x-y|^{n-2s}}\min\left\{\bigg[\frac{(1-|x|)(1-|y|)}{|x-y|^2}\bigg]^s,1\right\}.
\end{equation}
\end{lemma}

In this paper, the mollification operator: $J_\epsilon $ is needed:
\begin{equation}
J_\epsilon [u](x)=(j_\epsilon\ast u)(x)=\int_{B_{\epsilon}(x)} j_\epsilon(x-y)u(y)\ dy,
\end{equation}
where $j_\epsilon\in C_0^{\infty}(B_\epsilon)$ is the mollifier given by
\begin{equation}j_\epsilon(x)=\frac{1}{\epsilon^n}j\left(\frac{x}{\epsilon}\right).\end{equation}
Here, $j\in C_0^{\infty}(B_1)$ is  a positive smooth radially symmetric function supported in $B_1$ satisfying $\int_{\mathbb{R}^n}j(x)\ dx=1$.

In order to solve the problem, we need to introduce several weighted Sobolev space in both fractional sense and regular sense:
for $0<\alpha<2$, $1\leqslant p\leqslant\infty$ and $-1<r<1$, we define
\begin{definition}
  \begin{equation}
\mathscr{W}^{\alpha,p}_r(\Omega)=\{u\in \mathcal{L}_{2s}|(-\Delta)^{\alpha/2}u\in L^1_{\mathrm{loc}}(\Omega)\ \text{in}\ \mathcal{D}'(\Omega)\ \text{sense}, (-\Delta)^{\alpha/2}u\in L^p_r(\Omega)\}
\end{equation}
with norm: $\|u\|_{\mathscr{W}^{\alpha,p}_r(\Omega)}=\|u\|_{ \mathcal{L}_{2s}}+\|(-\Delta)^{\alpha/2}u\|_{ L^p_r(\Omega)}$.
\end{definition}

In sight of Theorem \ref{cor1}, the following definition is nature to characterise $\mathscr{W}^{\alpha,p}_r(\Omega)$ functions with, in some sense, zero boundary data:
\begin{definition}
\begin{equation}
\mathscr{W}^{\alpha,p}_{r,0}(\Omega)=\{u\in \mathscr{W}^{\alpha,p}_r(\Omega)|u=0\ \text{in}\ \mathbb{R}^n\backslash\Omega\ \text{and\ satisfies\ condition}\ \eqref{cd:un1}\}
\end{equation}
with norm: $\|u\|_{\mathscr{W}^{\alpha,p}_{r,0}(\Omega)}=\|u\|_{\mathscr{W}^{\alpha,p}_r(\Omega)}$.
\end{definition}

A definition of the weighted $W^{1,p}$ space is also needed to fit the above weighted fractional order spaces:
\begin{definition}
  \begin{equation}
W^{1,p}_r(\Omega)=\{u\in W^{1,1}_{\mathrm{loc}}(\Omega)| u\ \text{and}\ Du\in L^p_r(\Omega)\}
\end{equation}
with norm: $\|u\|_{W^{1,p}_r(\Omega)}=\|u\|_{ L^p_r(\Omega)}+\|Du\|_{ L^p_r(\Omega)}$.
\end{definition}
\begin{remark}
  In this paper, we only consider the case when $\Omega=B_1$.
\end{remark}
\begin{remark}
It is easy to verify that space $\mathscr{W}^{\alpha,p}_r(\Omega)$ and $W^{1,p}_r(\Omega)$ are Banach spaces, the completeness immediately comes from the completeness of $L^p$-spaces.
\end{remark}
\begin{remark}
  When we restrict $\Omega=B_1$, then space $\mathscr{W}^{\alpha,p}_{r,0}(B_1)$ is Banach. Indeed, Theorem \ref{thm:un} and Theorem \ref{cor2} implies:
  \begin{equation}
  \mathscr{W}^{\alpha,p}_{r,0}(B_1)=\{u=G\ast f|f\in L^p_r(B_1)\},
  \end{equation}
  one can define the equivalent norm as:
  \begin{equation}
    \|G\ast f\|_{\mathscr{W}^{\alpha,p}_{r,0}(B_1)}=\|f\|_{L^p_r(B_1)}.
  \end{equation}
\end{remark}

An essential observation of this paper is the embedding theorem between such spaces when $\Omega=B_1$(see Section \ref{emb}).

\section{Dirichlet problem for the fractional Laplacian}
In this section, we establish the existence and uniqueness theorem for the problem \eqref{pb1}

\begin{theorem}[Uniqueness]
\label{thm:un}
  Let $0<s<1$, $f=0$ and suppose  $u\in L^1(B_1)$ is a solution of \eqref{pb1} satisfying \eqref{cd:un1}, then $u=0$ a.e. $B_1$.
\end{theorem}
\begin{proof}
First we do the mollify, let $u_\epsilon(x)=J_\epsilon u(x)$. Then we know $u_\epsilon\in C_0^\infty(\mathbb{R}^n)$ satisfies the following equations
\begin{equation}
\left\{\begin{aligned}
(-\Delta)^s u_\epsilon&=0\ \ &\text{in}\ &B_{1-\epsilon}\\
u_\epsilon&=0\ \ &\text{in}\ &\mathbb{R}^n\backslash B_{1+\epsilon}
\end{aligned}\right.
\end{equation}

Then we can  represent $u$,  by using a rescaled version of Lemma~\ref{exi1}, as
\begin{equation}
u_\epsilon(x)
=\int_{\mathbb{R}^n\backslash B_{1-\epsilon}}P_{1-\epsilon}(x,y)u_\epsilon(y)\ dy,\ \text{for}\ x\in B_{1-\epsilon},
\end{equation}
where $P_r(x,y)$ is the rescaled Poisson kernel:
\begin{equation}
P_{r}(x,y)=c(n,s)\left(\frac{r^2-|x|^2}{|y|^2-r^2}\right)^s \frac{1}{|x-y|^n},\ \ \text{for}\ x\in B_r\ \text{and}\ y\in \mathbb{R}^n\backslash B_r,
\end{equation}
Now for a given $r\in (0,1)$, we choose $0<\epsilon<\frac{1-r}{3}$ and estimate $\|u_\epsilon\|_{L^\infty(B_r)}$. Indeed for $x\in B_r$,
\begin{equation}
\begin{aligned}
|u_\epsilon(x)|
&=c(n,s)\left|\int_{\mathbb{R}^n\backslash B_{1-\epsilon}} \left(\frac{(1-\epsilon)^2-|x|^2}{|y|^2-(1-\epsilon)^2}\right)^s \frac{u_\epsilon(y)}{|x-y|^n}\ dy\right|\\
&=c(n,s)\left|\int_{B_{1+\epsilon}\backslash B_{1-\epsilon}} \left(\frac{(1-\epsilon)^2-|x|^2}{|y|^2-(1-\epsilon)^2}\right)^s \frac{1}{|x-y|^n}\int_{B_1\backslash B_{1-2\epsilon}}u(z)j_\epsilon(y-z)\ dz dy\right|\\
&\leqslant c(n,s)\int_{B_1\backslash B_{1-2\epsilon}}\left[\int_{B_{1+\epsilon}\backslash B_{1-\epsilon}} \left(\frac{(1-\epsilon)^2-|x|^2}{|y|^2-(1-\epsilon)^2}\right)^s \frac{ j_\epsilon(y-z)}{|x-y|^n}\ dy\right] |u(z)|\ dz\\
&\leqslant c(n,s)\int_{B_1\backslash B_{1-2\epsilon}}\left[\int_{B_{1+\epsilon}\backslash B_{1-\epsilon}} \left(\frac{1}{(|y|-1+\epsilon)}\right)^s \frac{C\chi_{|y-z|\leqslant\epsilon}}{\epsilon^n(\frac{1-r}{3})^n}\ dy\right] |u(z)|\ dz\\
&\leqslant M(r)\frac{1}{\epsilon^n}\int_{B_1\backslash B_{1-2\epsilon}}\left[\int_{(B_{1+\epsilon}\backslash B_{1-\epsilon})\cap B_\epsilon(z)} \left(\frac{1}{|y|-1+\epsilon}\right)^s\ dy\right] |u(z)|\ dz\\
&\leqslant M(r)\frac{1}{\epsilon^s}\int_{B_1\backslash B_{1-2\epsilon}} |u(z)|\ dz.
\end{aligned}
\end{equation}
i.e.
\begin{equation}\label{ieq1}
\|u_\epsilon\|_{L^\infty(B_r)}\leqslant M(r)\frac{1}{\epsilon^s}\int_{B_1\backslash B_{1-2\epsilon}} |u(z)|\ dz.
\end{equation}

Noting the condition \eqref{cd:un1}, one obtain:
\begin{equation}\label{eq1}
\|u\|_{L^\infty(B_r)}=0,
\end{equation}
by letting $\epsilon\rightarrow0^+$ in \eqref{ieq1}.

Since $r$ in \eqref{eq1} is arbitrary, one can drive $r\rightarrow1^-$ in \eqref{eq1} and derive $\|u\|_{L^\infty(B_1)}=0$.
\end{proof}

\begin{theorem}[Existence]\label{cor2}
Let $0<s<1$  and suppose $f\in  L^1_s(B_1)$. Then $u=G\ast f$ solves the problem \eqref{pb1}. Furthermore, $u$ satisfies the uniqueness condition \eqref{cd:un1}.
\end{theorem}
\begin{proof}
We first show that such $u$ satisfies the condition \eqref{cd:un1}.

Indeed, for $\epsilon<\frac{1}{8}$, one calculates:
\begin{equation}
\begin{aligned}
  &\frac{1}{\epsilon^s}\int_{B_1\backslash B_{1-\epsilon}}|u(x)| dx\\
  \leqslant& \frac{1}{\epsilon^s}\int_{B_1\backslash B_{1-\epsilon}}\left[\int_{B_1}G(x,y)|f(y)|dy\right]dx\\
  =&\int_{B_1}\underbrace{\left[\frac{1}{\epsilon^s}\int_{B_1\backslash B_{1-\epsilon}}G(x,y)dx\right]}_{\text{denote\ as}\ A_\epsilon(y)}|f(y)|dy\\
\end{aligned}
\end{equation}
Utilizing Lemma \ref{est3}, one estimates $A_\epsilon$ as follow:
\begin{itemize}
  \item If $y\in B_1\backslash B_{1-2\epsilon}$
  \begin{equation}
    \begin{aligned}
    A_\epsilon (y)&\leqslant \frac{1}{\epsilon^s}\int_{B_1\backslash B_{1-4\epsilon}}G(x,y)dx\\
    &\leqslant\frac{C}{\epsilon^s}\left[\int_{(B_1\backslash B_{1-4\epsilon})\backslash B_{\frac{1-|y|}{2}}(y)}\frac{(1-|x|)^s(1-|y|)^s}{|x-y|^n}dx+\int_{B_{\frac{1-|y|}{2}}(y)}\frac{1}{|x-y|^{n-2s}}dx\right]\\
    &\leqslant \frac{C(1-|y|)^s}{\epsilon^s}\int_{0}^{4\epsilon}\int_{\partial B_{1-t}\backslash B_{\frac{1-|y|}{2}}(y) }\frac{t^s}{|x-y|^n}d\sigma_xdt+C\frac{(1-|y|)^{2s}}{\epsilon^s}\\
    &\leqslant \frac{C(1-|y|)^s}{\epsilon^s}\left[\int_{0}^{2(1-|y|)}\frac{t^s}{1-|y|}dt +\int_{2(1-|y|)}^{4\epsilon}\frac{t^s}{t-(1-|y|)}dt \right]
    +C\frac{(1-|y|)^{2s}}{\epsilon^s} \\
    &\leqslant C(1-|y|)^s.
    \end{aligned}
  \end{equation}
  \item If  $y\in B_{1-2\epsilon}$
  \begin{equation}
  \begin{aligned}
    A_\epsilon (y)&\leqslant \frac{C}{\epsilon^s}\int_{B_1\backslash B_{1-\epsilon}}\frac{(1-|x|)^s(1-|y|)^s}{|x-y|^n}dx\\
    &=\frac{C(1-|y|)^s}{\epsilon^s}\int_{0}^\epsilon\int_{\partial B_{1-t}}\frac{t^s}{|x-y|^n}d\sigma_xdt\\
    &\leqslant \frac{C(1-|y|)^s}{\epsilon^s}\int_{0}^\epsilon\frac{t^s}{1-|y|}dt\quad\left(\text{since}\ \operatorname{dist}(y,\partial B_{1-t})\geqslant\frac{1-|y|}{2}\right)\\
    &\leqslant C\frac{\epsilon}{1-|y|}(1-|y|)^s.
    \end{aligned}
  \end{equation}
\end{itemize}
Therefore, $A_\epsilon(y)\leqslant C(1-|y|)^s$ and $A_\epsilon(y)\rightarrow 0$ as $\epsilon\rightarrow 0$.

Now, we know $A_\epsilon(y)|f(y)|\leqslant \delta(y)^s|f(y)|\in L^1(B_1)$ and $A_\epsilon(y)|f(y)|\rightarrow 0$ as $\epsilon\rightarrow 0$, the Lebesgue convergence theorem assures that
\begin{equation}
  \limsup_{\epsilon\rightarrow 0}\frac{1}{\epsilon^s}\int_{B_1\backslash B_{1-\epsilon}}|u(x)| dx\leqslant \lim_{\epsilon\rightarrow 0}\int_{B_1}A_\epsilon(y)|f(y)|dy=0.
\end{equation}
I.e. $u$ satisfies the uniqueness condition \eqref{cd:un1}.

The above calculation also gives us a basic estimate:
\begin{equation}
\label{est4}
  \|G\ast f\|_{L^1(B_1)}\leqslant C\|f\|_{L^1_s(B_1)}.
\end{equation}

Since $f\in  L^1_s(B_1)$, we can choose $f_k\in C^\infty(B_1)\cap L^1_s(B_1)$ that converges to $f$ in $ L^1_s(B_1)$ as $k\rightarrow\infty$.
According to Theorem~\ref{exi0}, there exists a corresponding solution sequence $\{u_k=G\ast f_k\}$.
\eqref{est4} implies that $u_k$ converges to $u$ in $L^1(B_1)$ as $k\rightarrow\infty$. Hence, $u$ as the distributional limit of $u_k$ satisfies \eqref{pb1}.
\end{proof}
\begin{remark}
Theorem~\ref{cor1} is a direct corollary of Theorem \ref{thm:un} and Theorem \ref{cor2}
\end{remark}
\section{Embedding of $\mathscr{W}^{\alpha,p}_{r,0}(B_1)$ into other spaces}\label{emb}
In this section, we want to give some embedding estimates for our newly defined `weighted Sobolev' spaces in Section \ref{2}. While, these estimates are essential during our proof of the solvability theorem, they also shows that the spaces we define in Section \ref{2} do make sense. Indeed, we can see these spaces appears to satisfy a similar embedding property as the ordinary Sobolev space with only a little bit adjustment.

From Theorem \ref{cor1}, we know: for $0<r\leqslant s$, $u\in\mathscr{W}^{2s,p}_{r,0}(B_1)$ is equivalent to saying $u=G\ast f$ for some $f\in L^p_r$. Hence to show the embedding theorems, we only need to estimate the Green's potential $G\ast f$.

To begin with, we need a technical lemma:
\begin{lemma}\label{lem0}
For $x,y\in B_1$, $0<\beta<1$ and $-\beta\leqslant\alpha\leqslant\beta$.
\begin{equation}
\min\left\{\bigg[\frac{(1-|x|)(1-|y|)}{|x-y|^2}\bigg]^\beta,1\right\}\leqslant 4\left(\frac{1-|y|}{1-|x|}\right)^\alpha.
\end{equation}
\end{lemma}
\begin{proof}
Notice the symmetry of $x$ and $y$, we only need to prove for the case $\alpha\geqslant0$.
Notice also that
\begin{equation}
\min\left\{\bigg[\frac{(1-|x|)(1-|y|)}{|x-y|^2}\bigg]^\beta,1\right\}\leqslant \min\left\{\bigg[\frac{(1-|x|)(1-|y|)}{|x-y|^2}\bigg]^\alpha,1\right\}
\end{equation}
when $0\leqslant\alpha\leqslant\beta$. It suffices to prove for the case $\alpha=\beta$.

Consider the following two cases:
\begin{enumerate}
\item If $1-|x|\geqslant 2(1-|y|)$, then $|x-y|\geqslant (1-|x|)-(1-|y|)\geqslant \frac{1}{2}(1-|x|)$, and hence
\begin{equation}
\min\left\{\bigg[\frac{(1-|x|)(1-|y|)}{|x-y|^2}\bigg]^\alpha,1\right\}
\leqslant \bigg[\frac{(1-|x|)(1-|y|)}{|x-y|^2}\bigg]^\alpha\leqslant 4\bigg(\frac{1-|y|}{1-|x|}\bigg)^\alpha.
\end{equation}
\item If $1-|x|\leqslant 2(1-|y|)$, then
\begin{equation}
\min\left\{\bigg[\frac{(1-|x|)(1-|y|)}{|x-y|^2}\bigg]^\alpha,1\right\}
\leqslant  1\leqslant 2^\alpha\left(\frac{1-|y|}{1-|x|}\right)^\alpha\leqslant 4\left(\frac{1-|y|}{1-|x|}\right)^\alpha.
\end{equation}
\end{enumerate}
\end{proof}
\subsection{Embedding results}
\begin{theorem}\label{est1}
Let $0<s<1$, $-s\leqslant r\leqslant s$, $1\leqslant p<\infty$, and suppose $u=G\ast f$ for $f\in  L^p_r(B_1)$.
 Then we  have the following estimates:
\begin{enumerate}
\item If $p=1$, then for any $q$ satisfying $1\leqslant q<\frac{n}{n-2s}$,
$u\in  L^q_r(B_1)$ with
\begin{equation}
\|u\|_{ L^q_r(B_1)}\leqslant C\|f\|_{ L^1_r(B_1)},
\end{equation}
where $C$ is a positive constant depending only on $n,s,q$.
\item If $1<p<\frac{n}{2s}$, then $u\in  L^{\frac{np}{n-2sp}}_r(B_1)$ with
\begin{equation}
\|u\|_{ L^{\frac{np}{n-2sp}}_r(B_1)}\leqslant C\|f\|_{ L^p_r(B_1)},
\end{equation}
where $C$ is a positive constant depending only on $n,s,p$.
\item If $p>\frac{n}{2s}$, then $u\in  L^\infty_r(B_1)$ with
\begin{equation}
\|u\|_{ L^\infty_r(B_1)}\leqslant C\|f\|_{ L^p_r(B_1)},
\end{equation}
where $C$ is a positive constant depending only on $n,s,p$.
\end{enumerate}
\end{theorem}
\begin{proof}
Hence \begin{equation}
G(x,y)\leqslant \underbrace{\frac{C_0}{|x-y|^{n-2s}}\min\left\{\bigg[\frac{(1-|x|)(1-|y|)}{|x-y|^2}\bigg]^s,1\right\}\leqslant 4C_0\frac{(1-|y|)^{r}}{|x-y|^{n-2s}(1-|x|)^{r}}}_{\mathrm{by\ Lemma~\ref{lem0}}}.
\end{equation}
For $x\in B_1$, we have
\begin{equation}
\begin{aligned}
(1-|x|)^{r}|G\ast f(x)|&=\bigg|(1-|x|)^{r}\int_{B_1}G(x,y)f(y)\ dy \bigg|\\
&\leqslant (1-|x|)^{r}\int_{B_1}G(x,y)|f(y)|\ dy \\
&\leqslant 4C_0\int_{B_1}\frac{(1-|y|)^{r}}{|x-y|^{n-2s}}|f(y)|\ dy.
\end{aligned}
\end{equation}
The Hardy-Littlewood-Sobolev inequality gives directly the desired result.
\end{proof}

\begin{remark}
Theorem~\ref{est1} indicates that for $0<s<1$, $-s\leqslant r\leqslant s$ and $1\leqslant p<\infty$, the following embedding holds:
\begin{equation}
\mathscr{W}^{2s,p}_{r,0}(B_1)\hookrightarrow\left\{
\begin{aligned}
& L^{q}_r(B_1)&p=1,\quad1\leqslant q<\frac{n}{n-2s},\\
& L^{\frac{np}{n-2sp}}_r(B_1)&1<p<\frac{n}{2s},\\
& L^{\infty}_r(B_1)&p>\frac{n}{2s}.
\end{aligned}\right.
\end{equation}
\end{remark}

\begin{theorem}\label{est2}
Let $\frac{1}{2}<s<1$, $1-s\leqslant r\leqslant s$, $1\leqslant p\leqslant\infty$, and suppose $u=G\ast f$ for  $f\in  L^p_r(B_1)$.
 Then we  have the following estimates:
\begin{enumerate}
\item If $p=1$, then for any $q$ satisfying $1\leqslant q<\frac{n}{n-2s+1}$,
$|\nabla u|\in  L^q_r(B_1)$ with
\begin{equation}
\||\nabla u|\|_{ L^q_r(B_1)}\leqslant C\|f\|_{ L^1_r(B_1)},
\end{equation}
where  $C$ is a positive constant depending only on $n,s,q$.
\item If $1<p<\frac{n}{2s-1}$, then $|\nabla u|\in  L^{\frac{np}{n-(2s-1)p}}_r(B_1)$ with
\begin{equation}
\||\nabla u|\|_{ L^{\frac{np}{n-(2s-1)p}}_r(B_1)}\leqslant C\|f\|_{ L^p_r(B_1)},
\end{equation}
where $C$ is a positive constant depending only on $n,s,p$.
\item If $p>\frac{n}{2s-1}$, then $|\nabla u|\in  L^\infty_r(B_1)$ with
\begin{equation}
\||\nabla u|\|_{ L^\infty_r(B_1)}\leqslant C\|f\|_{ L^p_r(B_1)},
\end{equation}
where $C$ is a positive constant depending only on $n,s,p$.
\end{enumerate}
\end{theorem}
\begin{proof}
We first calculate the gradient of the Green's function.
\begin{equation}
\frac{\nabla_x G(x,y)}{\kappa(n,s)}=-\frac{(n-2s)(x-y)}{|x-y|^{n-2s+2}}\int_0^{\rho(x,y)}\frac{t^{s-1}}{(1+t)^\frac{n}{2}}\ dt+\frac{\nabla_x\rho(x,y)}{|x-y|^{n-2s}}\frac{(\rho(x,y))^{s-1}}{(1+\rho(x,y))^{\frac{n}{2}}},
\end{equation}
where
\begin{equation}
\begin{aligned}
\nabla_x\rho(x,y)&=-\frac{2x(1-|y|^2)}{|x-y|^2}-\frac{2(x-y)(1-|x|^2)(1-|y|^2)}{|x-y|^4}\\
&=-\frac{2x}{(1-|x|^2)}\rho(x,y)-\frac{2(x-y)}{|x-y|^2}\rho(x,y).
\end{aligned}
\end{equation}
Hence
\begin{equation}
\begin{aligned}
\frac{\nabla_x G(x,y)}{\kappa(n,s)}&=\frac{(y-x)}{|x-y|^{n-2s+2}}\underbrace{\left[\int_0^{\rho(x,y)}\frac{(n-2s)t^{s-1}}{(1+t)^\frac{n}{2}}\ dt+\frac{2(\rho(x,y))^s}{(1+\rho(x,y))^{\frac{n}{2}}}\right]}_{{\rm denote\ as}\ I_1(x,y)}\\
&\ \ \ \ \ \ \ \ \ \ \ \ -\frac{2x}{(1-|x|^2)|x-y|^{n-2s}}\underbrace{\frac{(\rho(x,y))^{s}}{(1+\rho(x,y))^{\frac{n}{2}}}}_{{\rm denote\ as}\ I_2(x,y)}.
\end{aligned}\end{equation}

Hence,  we have
\begin{equation}\label{2.1}
\frac{|\nabla_x G(x,y)|(1-|x|)^r}{\kappa(n,s)}\leqslant \frac{(1-|x|)^r}{|x-y|^{n-2s+1}}I_1(x,y)+\frac{4I_2(x,y)}{(1-|x|)^{1-r}|x-y|^{n-2s}},
\end{equation}
where for both $i=1,2$, we know
\begin{equation}
I_i\leqslant \frac{C_i}{4}\min\{(\rho(x,y))^s,1\}\leqslant C_i\min\bigg\{\bigg(\frac{(1-|x|)(1-|y|)}{|x-y|^2}\bigg)^s,1\bigg\}.
\end{equation}

Apply Lemma~\ref{lem0} for $I_1$, we have:
\begin{equation}
I_1(x,y)\leqslant 4C_1 \left(\frac{1-|y|}{1-|x|}\right)^r.
\end{equation}
For $I_2$, we first observe:
\begin{equation}
I_2(x,y)\leqslant C_2\bigg(\frac{(1-|x|)(1-|y|)}{|x-y|^2}\bigg)^{\frac{1}{2}}\min\bigg\{ \bigg(\frac{(1-|x|)(1-|y|)}{|x-y|^2}\bigg)^{s-\frac{1}{2}},1\bigg\},
\end{equation}
And since $|r-\frac{1}{2}|\leqslant s-\frac{1}{2}$, apply Lemma~\ref{lem0}
\begin{equation}
I_2(x,y)\leqslant 4C_2\bigg(\frac{(1-|x|)(1-|y|)}{|x-y|^2}\bigg)^{\frac{1}{2}}\bigg(\frac{1-|y|}{1-|x|}\bigg)^{r-\frac{1}{2}}= 4C_2\frac{(1-|x|)^{1-r}(1-|y|)^r}{|x-y|}.
\end{equation}

Combining this two observations with \eqref{2.1}, we get
\begin{equation}
|\nabla_x G(x,y)|(1-|x|)^{r}\leqslant C_3\frac{(1-|y|)^r}{|x-y|^{n-2s+1}}.
\end{equation}
Then for given $x\in B_1$, we have
\begin{equation}
\begin{aligned}
|\nabla u(x)|(1-|x|)^{r}&=\left|(1-|x|)^r\int_{B_1}\nabla_xG(x,y)f(y)\ dy\right|\\
&\leqslant \int_{B_1}|\nabla_xG(x,y)|(1-|x|)^r|f(y)|\ dy\\
&\leqslant C_3\int_{B_1}\frac{(1-|y|)^r|f(y)|}{|x-y|^{n-2s+1}}\ dy.
\end{aligned}
\end{equation}
The Hardy-Littlewood-Sobolev inequality gives directly the desired result.
\end{proof}

\begin{remark}
Theorem~\ref{est1} and~\ref{est2} indicates that for $\frac{1}{2}<s<1$, $1-s\leqslant r\leqslant s$ and $1\leqslant p\leqslant\infty$, the following embedding holds:
\begin{equation}
\mathscr{W}^{2s,p}_{r,0}(B_1)\hookrightarrow\left\{
\begin{aligned}
&W^{1,q}_r(B_1)&p=1,\quad1\leqslant q<\frac{n}{n-2s+1},\\
&W^{1,\frac{np}{n-(2s-1)p}}_r(B_1)&1<p<\frac{n}{2s-1},\\
&W^{1,\infty}_r(B_1)&p>\frac{n}{2s-1}.
\end{aligned}\right.
\end{equation}
\end{remark}

The above embedding results lead to a  continuity result:
\begin{theorem}\label{C1+}
Let $\frac{1}{2}<s<1$, $1-s\leqslant r\leqslant s$, $p>\frac{n}{2s-1}$ and $u\in \mathscr{W}^{2s,p}_{r,0}(B_1)$.
 Then $u\in C^{1,\gamma}_{\rm loc}(B_1)$ for $0<\gamma\leqslant2s-1-\frac{n}{p}$.
\end{theorem}

\begin{proof}
Fix an $\epsilon\in(0,1)$, define
\begin{equation}
f_\epsilon(x)=f(x)\chi_{B_{1-\epsilon}}(x).
\end{equation}
Since $\delta^{r}f\in L^p(B_1)$, we derive $f\in L^p(B_{1-\epsilon})$ and hence $f_\epsilon\in L^p(\mathbb{R}^n)$. Now, define
\begin{equation}
v_\epsilon(x)=c(n,s)\int_{\mathbb{R}^n}\frac{f_\epsilon(y)}{|x-y|^{n-2s}}\ dy.
\end{equation}
We know $v_\epsilon\in C^{1,\gamma}(\mathbb{R}^n)$ and $(-\Delta)^sv_\epsilon=f_\epsilon=f$ in $\mathcal{D}'(B_{1-\epsilon})$.

Hence, $(-\Delta)^s(u-v_\epsilon)=0$ in $\mathcal{D}'(B_{1-\epsilon})$. This indicates $u-v_\epsilon\in C^\infty(B_{1-2\epsilon})$\cite{li2020non}, and consequently
\begin{equation}
u\in C^{1,\gamma}(B_{1-2\epsilon}).
\end{equation}
\end{proof}
\begin{theorem}\label{C0+}
Let $0<r<1$  and $u\in W^{1,\infty}_r(B_1)$.
 Then $u\in C^{1-r}(\overline{B_1})$.
\end{theorem}
\begin{proof}
It suffices to show $|u(x)-u(y)|\leqslant C|x-y|^{1-r}$ for all $x,y\in B_1$.

Indeed, from $u\in W^{1,\infty}_r(B_1)$, we know:
\begin{equation}
|\nabla u(x)|\leqslant M (1-|x|)^{-r}\quad\text{a.e.}\quad B_1
\end{equation}

Without loss of generality, we assume $|x|\geqslant|y|$.
Consider the following two cases:
\begin{enumerate}
\item If $|x-y|\leqslant1-|x|\leqslant 1-|y|$, then
\begin{equation}
|u(x)-u(y)|\leqslant M|x-y|^{1-r}.
\end{equation}
\item If $1-|x|\leqslant|x-y|<1-|y|$,  define
\begin{gather*}
x'=\frac{|y|}{|x|}x.
\end{gather*}
Then, $|x'|=|y|>1-|x-y|$, $|x'-y|\leqslant|x-y|$ and $|x-x'|\leqslant|x-y|$.
Consequently:
\begin{equation}
|u(x')-u(y)|\leqslant M|x-y|^{-r}|x'-y|\leqslant M|x-y|^{1-r}
\end{equation}
\begin{equation}
\begin{aligned}
|u(x)-u(x')|\leqslant M\int_{1-|x|}^{1-|y|}t^{-r}\ dt&\leqslant\frac{M}{1-r}[(1-|y|)^{1-r}-(1-|x|)^{1-r}]\\
&\leqslant\frac{M}{1-r}(|x|-|y|)^{1-r}\leqslant\frac{M}{1-r}|x-y|^{1-r}
\end{aligned}
\end{equation}
and $|u(x)-u(y)|\leqslant M|x-y|^{1-r}+\frac{M}{1-r}|x-y|^{1-r}\leqslant C|x-y|^{1-r}$.
\item If $1-|x|<1-|y|\leqslant |x-y|$,  define
\begin{gather*}
x'=\frac{1-|x-y|}{|x|}x,\qquad y'=\frac{1-|x-y|}{|y|}y.
\end{gather*}
Then $1-|x'|=1-|y'|=|x-y|$ and $|x'-y'|\leqslant|x-y|$. Consequently
\begin{equation}
|u(x')-u(y')|\leqslant M|x-y|^{-r}|x'-y'|\leqslant M|x-y|^{1-r}.
\end{equation}
\begin{equation}
|u(x)-u(x')|\leqslant M\int_{1-|x|}^{|x-y|}t^{-r}\ dt\leqslant\frac{M}{1-r}[(|x-y|)^{1-r}-(1-|x|)^{1-r}]\leqslant\frac{M}{1-r}|x-y|^{1-r}.
\end{equation}
\begin{equation}
\text{Similarly}\quad
|u(y)-u(y')|\leqslant \frac{M}{1-r}|x-y|^{1-r}.
\end{equation}
and $|u(x)-u(y)|\leqslant M|x-y|^{1-r}+\frac{M}{1-r}|x-y|^{1-r}+\frac{M}{1-r}|x-y|^{1-r}\leqslant C|x-y|^{1-r}$.
\end{enumerate}
\end{proof}

The following compact embedding is also true:
\begin{theorem}\label{cpt}
Let $\frac{1}{2}<s<1$, $1-s\leqslant r\leqslant s$, $1\leqslant p<\infty$. Then
$\mathscr{W}^{2s,p}_{r,0}(B_1)\subset\subset W^{1,p}_r(B_1)$
\end{theorem}
\begin{proof}
Observe that $u\in \mathscr{W}^{2s,p}_{r,0}(B_1)$ indicates that
\begin{equation}
u(x)=G\ast g=\int_{B_1}G(x,y)g(y)\ dy\chi_{B_1}(x)\quad\text{for\ some}\quad g\in  L^p_r(B_1)
\end{equation}
Hence, it suffices to prove that each of the following integral operators $\{\Psi_i:L^p(B_1)\rightarrow L^p(B_1)\}_{i=0}^n$ are compact. Here $\{\Psi_i\}$ are defined as:
\begin{equation}
\begin{aligned}
\Psi_0[f](x)&=\int_{B_1}\left(\frac{1-|x|}{1-|y|}\right)^rG(x,y)f(y)\ dy,\\
\Psi_j[f](x)&=\int_{B_1}\left(\frac{1-|x|}{1-|y|}\right)^r\partial_{x_j}G(x,y)f(y)\ dy,\quad j=1,2,\cdots,n.
\end{aligned}
\end{equation}
Take $\Psi_1$ for example.
Write $K(x,y)=\left(\frac{1-|x|}{1-|y|}\right)^r\partial_{x_j}G(x,y)$.

It's clear that $K:B_1\times B_1\rightarrow \mathbb{R}$ is continuous whenever $x\neq y$.
As we have derived in Theorem~\ref{est2}, we also have: $|K(x,y)|\leqslant C\frac{1}{|x-y|^{n-2s+1}}$.
Define:
\begin{equation}
K_N(x,y)=\left\{
\begin{aligned}
&K(x,y)&\text{if}\quad K(x,y)\leqslant N,\\
&\text{continuously\ extended,\ and}\leqslant N&\text{if}\quad K(x,y)>N.
\end{aligned}
\right.
\end{equation}
Then clearly, $K_N\in C(B_1\times B_1)\cap L^\infty(B_1\times B_1)$.

Then, it is a classical result that: for each $N$, the integral operator $\Omega_N: L^p(B_1)\rightarrow L^p(B_1)$ is a compact operator, where
\begin{equation}
\Omega_N[f](x):=\int_{B_1} K_N(x,y)f(y)\ dy.
\end{equation}

Another important observation is that
\begin{equation}
\begin{aligned}
\|\Omega_N[f]-\Psi_1[f]\|_{L^p(B_1)}
&\leqslant C\left\|\int_{B_1}|K_N(x,y)-K(x,y)||f(y)|\ dy\right\|_{L^p(B_1)}\\
&\leqslant C\left\|\int_{B_1}\left|\frac{1}{|x-y|^{n-2s}}\chi_{|x-y|^{2s-n}>N}\right||f(y)|\ dy\right\|_{L^p(B_1)}\\
&\leqslant C\||x|^{2s-n}\chi_{|x|^{2s-n}>N}\|_{L^1(B_2)}\|f\|_{L^p(B_1)}
\end{aligned}
\end{equation}
Since $\||x|^{2s-n}\chi_{|x|^{2s-n}>N}\|_{L^1(B_2)}\rightarrow 0$ as $N\rightarrow\infty$, we know $\Omega_N[f]$ converges uniformly (does not depend on $f$) to $\Psi_1[f]$ in $L^p(B_1)$ whenever $\|f\|_{L^p(B_1)}$ is uniformly bounded.
Hence, a direct diagonal method gives the compactness of $\Psi_1$.
\end{proof}

\section{Dirichlet problem for equations with general lower order terms}
We are now ready to prove Theorem \ref{exi}

\subsection{A maximum principle for $C^1$-solution}
\begin{theorem}\label{maxp}
Let $\frac{1}{2}<s<1$, $\vec b, c\in L^\infty(B_1)$ and $c\geqslant 0$ in $B_1$. Suppose  $u\in C(\mathbb{R}^n)\cap C^1(B_1)\cap\mathcal{L}_{2s}$ satisfies
\begin{equation}\label{pb7}
\left\{
\begin{aligned}
(-\Delta)^su+\vec b\cdot\nabla u+cu\geqslant0\quad&\text{in}\quad\mathcal{D}'(B_1),\\
u\geqslant0\quad&\text{in}\quad \mathbb{R}^n\backslash B_1.
\end{aligned}\right.
\end{equation}
Then $u\geqslant 0$ in $B_1$.
\end{theorem}
\begin{proof}
We prove by contradiction. Denote $u_\delta(x)=J_\delta u(x)$ be the mollification of $u$.
Define:
\begin{equation}
M_0=\int_{\mathbb{R}^n\backslash B_2}\frac{C_{n,s}}{|y|^{n-2s}}\ dy\leqslant \int_{\mathbb{R}^n\backslash B_2}\frac{C_{n,s}}{|x-y|^{n-2s}}\ dy\quad\text{for}\quad x\in B_1.
\end{equation}

Suppose that
\begin{gather*}
\min_{\overline{B_1}}u=-2m<0\quad\text{and\ let}\quad K=\{x\in \mathbb{R}^n|u(x)=-2m\}\subset B_1.
\end{gather*}
\begin{enumerate}
\item[\textit{Step 1.}]We show that for sufficiently small $\delta$ the minimum point of $u_\delta$ must be near $K$.

Indeed, from the continuity of $u$, $K$ is closed and compact.
Let
\begin{equation}
\rho_0=\frac{1}{2}\mathrm{dist}(K,\partial B_1)>0
\end{equation}
For $0<\epsilon<\rho_0$, define the $\epsilon$-neighborhood of $K$ as:
\begin{gather*}
K_\epsilon=\{x\in\mathbb{R}^n|\mathrm{dist}(x,K)<\epsilon\}\subset B_{1-\rho_0}.
\end{gather*}
Observe that $u$ and $\nabla u$ are uniformly continuous in $B_{1-\rho_0}$ with $u=-2m$ and $\nabla u=0$ in $K$. Then, there exists $\epsilon_0\in (0,\frac{\rho_0}{3})$ such that for any $x\in K_{3\epsilon_0}$:
\begin{equation}
u(x)\leqslant0\quad \text{and}\quad|\nabla u(x)|<\frac{M_0m}{\|\vec b\|_{L^\infty(B_1)}}.
\end{equation}

On the one hand, there exist $\sigma\in(0,m)$ such that
\begin{equation}
H:=\{x\in\mathbb{R}^n|u(x)<-2m+\sigma\}\subset K_{\epsilon_0}.
\end{equation}
On the other hand, $H$ as an open set containing $K$ satisfies
\begin{equation}
H\supset K_{\delta_0}\quad\text{for\ some}\ 0<\delta_0<\epsilon_0.
\end{equation}
Consequently, if $0<\delta<\delta_0$ and $x_0$ is a minimum point of $u_\delta$ in $\overline{B_2}$, then one must have $x_0\in K_{2\epsilon_0}$. Further $u_\delta(x_0)\leqslant -2m+\sigma<m$.
\item[\textit{Step 2.}]We consider $u_\delta$ at its minimum point and derive a contradiction.

Leting $\delta<\delta_0$ and mollifying the inequality \eqref{pb7}, we have
\begin{equation}\label{ieq4}
(-\Delta)^su_\delta(x)+J_\delta(\vec b\cdot\nabla u+cu)(x)\geqslant0\quad\text{for\ all}\quad x\in B_{1-\delta}\supset B_{1-\rho_0}.
\end{equation}
Taking $x=x_0\in K_{2\epsilon_0}\subset B_{1-\rho_0}$ to be the minimum point of $u_\delta$ in $B_2$, we know:
\begin{enumerate}
    \item $u(x)\leqslant0$ for $x\in B_\delta(x_0)\subset B_{3\epsilon_0}$, and hence $J_\delta(cu)(x_0)\leqslant 0$.
    \item $|\vec b(x)||\nabla u(x)|<M_0m$ for $x\in B_\delta(x_0)\subset B_{3\epsilon_0}$, and hence $J_\delta(\vec b\cdot\nabla u)(x_0)<M_0m$.
    \item  $\displaystyle(-\Delta)^s u_\delta(x_0)=C_{n,s}\mathrm{P.V.}\int_{\mathbb{R}^n}\frac{u(x_0)-u(y)}{|x_0-y|^{n-2s}}\ dy< -\int_{\mathbb{R}^n\backslash B_2}\frac{mC_{n,s}}{|x_0-y|^{n-2s}}\ dy\leqslant-M_0m$.
\end{enumerate}
Combining this three facts and \eqref{ieq4}, we derives a contradiction:
\begin{equation}
0\leqslant(-\Delta)^su_\delta(x_0)+J_\delta(\vec b\cdot\nabla u)(x_0)+J_\delta(cu)(x_0)<-M_0m+M_0m+0=0.
\end{equation}
\end{enumerate}
As a consequence of the above contradiction, $u\geqslant 0$ in $B_1$.
\end{proof}
\subsection{The uniqueness}
\begin{theorem}\label{uniq}
Let $\frac{1}{2}<s<1$, $1-s\leqslant r\leqslant s$ and suppose  $u\in \mathscr{W}^{2s,1}_{r,0}(B_1)$ is a solution of \eqref{pb3} for $f=0$, where $\vec b, c\in L^\infty(B_1)$, and $c\geqslant 0$ in $B_1$.
Then $u=0$ in $B_1$.
\end{theorem}

\begin{proof}
Since $f=0$, we can rewrite \eqref{pb3} as
\begin{equation}\label{pb4}
\left\{
\begin{aligned}
(-\Delta)^s u&=-(\vec b\cdot\nabla u+cu)\ \ \ &\text{in}\ &\mathcal{D}'(B_1),\\
u&=0\ \ \ &\text{in}\ &\mathbb{R}^n\backslash B_1,
\end{aligned}
\right.
\end{equation}
Since $u\in \mathscr{W}^{2s,1}_{r,0}(B_1)$, Theorem~\ref{est1} and Theorem~\ref{est2} ensures that
\begin{gather}
\||\nabla u|\|_{ L^q_r(B_1)}<\infty\quad \text{and}\quad
\|u\|_{ L^q_r(B_1)}<\infty,
\end{gather}
for $1<q<\frac{n}{n-2s+1}$.

Hence, $\vec b\cdot\nabla u+cu\in  L^q_r(B_1)$, $u\in \mathscr{W}^{2s,q}_{r,0}(B_1)$.
Theorem \ref{est1} and \ref{est2} once more gives:
\begin{gather}
\||\nabla u|\|_{ L^{\frac{nq}{n-(2s-1)q}}_r(B_1)}<\infty\quad \text{and}\quad
\|u\|_{ L^{\frac{nq}{n-(2s-1)q}}_r(B_1)}<\infty.
\end{gather}

Iterating this procedure, we get $u\in \mathscr{W}^{2s,\frac{nq}{n-k(2s-1)q}}_{r,0}(B_1)$ for all integer $k<\frac{n}{(2s-1)q}$, and finally $u\in \mathscr{W}^{2s,\infty}_{r,0}(B_1)$.

Then, by Theorem~\ref{C1+} and~\ref{C0+}  $u\in C^{1-r}(\overline{B_1})\cap C^{1,2s-1}_{\rm loc}(B_1)$. The  maximum principle (Theorem~\ref{maxp}) shows that $u=0$ in $B_1$.
\end{proof}

\subsection{Apriori Estimate}
\begin{lemma}\label{apest1}
Let $\frac{1}{2}<s<1$, $1-s\leqslant r\leqslant s$, $1\leqslant p<\infty$ and suppose  $u\in \mathscr{W}^{2s,p}_{r,0}(B_1)$ is a solution of \eqref{pb3}, for $f\in  L^p_r(B_1)$. Assume that $\|\vec b\|_{L^\infty(B_1)}+\|c\|_{L^\infty(B_1)}\leqslant\Lambda$.
Then
\begin{equation}
\|(-\Delta)^su\|_{ L^p_r(B_1)}\leqslant
\Lambda\{\|u\|_{W^{1,p}_r(B_1)}+\|f\|_{ L^p_r(B_1)}\}.
\end{equation}
\end{lemma}
\begin{proof}
Trivial.
\end{proof}

\begin{theorem}\label{apest4}
Let $\frac{1}{2}<s<1$, $1-s\leqslant r\leqslant s$, $1\leqslant p<\infty$ and suppose  $u\in \mathscr{W}^{2s,p}_{r,0}(B_1)$ is a solution of \eqref{pb3}, for $f\in  L^p_r(B_1)$. Assume that $\|\vec b\|_{L^\infty(B_1)}+\|c\|_{L^\infty(B_1)}\leqslant\Lambda$.
Then
\begin{equation}
\|(-\Delta)^su\|_{ L^p_r(B_1)}\leqslant C\|f\|_{ L^p_r(B_1)},
\end{equation}
where $C$ is a positive constant depending only on $n,s,p,r$  and $\Lambda$.
\end{theorem}
\begin{proof}
We prove by contradiction.

Suppose that the estimate is not true. Then there exists a sequence of $u_k$, $\vec b_k$, $c_k$ and $f_k$ with $u_k\in \mathscr{W}^{2s,p}_{r,0}(B_1)$, $\|\vec b_k\|_{L^\infty(B_1)}+\|c_k\|_{L^\infty(B_1)}\leqslant\Lambda$ and $f_k\in  L^p_r(B_1)$ such that
  \begin{equation}
  (-\Delta)^su_k+\vec b_k\cdot\nabla u_k+c_ku_k=f_k\quad\text{in}\ \mathcal{D}'(B_1),
  \end{equation}
  but $\|(-\Delta)^su_k\|_{ L^p_r(B_1)}\geqslant k\|f_k\|_{ L^p_r(B_1)}$.
  By the linearity, we can also assume $\|u_k\|_{W^{1,p}_r(B_1)}=1$.

  Then, we estimate from Theorem \ref{apest1} that,
  \begin{equation}
\|(-\Delta)^su_k\|_{ L^p_r(B_1)}\leqslant
\Lambda\{\|u_k\|_{W^{1,p}_r(B_1)}+\|f_k\|_{ L^p_r(B_1)}\}\leqslant\Lambda+\frac{\Lambda}{k}\|(-\Delta)^su_k\|_{ L^p_r(B_1)}.
  \end{equation}
Then we know $\|(-\Delta)^su_k\|_{ L^p_r(B_1)}\leqslant 2\Lambda$ for $k\geqslant 2\Lambda$. Consequently, $\|u_k\|_{\mathscr{W}^{2s,p}_{r,0}(B_1)}\leqslant C$ and $f_k\rightarrow 0$ in $ L^p_r(B_1)$ as $k\rightarrow\infty$.
By weak compactness, we can choose subsequences and assume:
\begin{gather*}
u_k\stackrel{w}{\longrightarrow}u_0\quad \text{in}\ \mathscr{W}^{2s,p}_r(B_1),\qquad
\vec b_k\stackrel{w\ast}{\longrightarrow}\vec b_0\quad\text{and}\quad
c_k\stackrel{w\ast}{\longrightarrow}c_0\quad\text{in}\ L^\infty(B_1)\qquad\text{as}\quad k\rightarrow\infty.
\end{gather*}
Hence, $u_0$ satisfies
\begin{equation}\label{pb5}
\left\{
\begin{aligned}
(-\Delta)^s u_0+\vec b_0\cdot\nabla u_0+c_0u_0&=0\ \ \ &\text{in}\ &\mathcal{D}'(B_1),\\
u_0&=0\ \ \ &\text{in}\ &\mathbb{R}^n\backslash B_1,
\end{aligned}
\right.
\end{equation}

By compact embedding (see Theorem~\ref{cpt}),
$u_k\rightarrow u_0$ in $W^{1,p}_r(B_1)$ as $k\rightarrow\infty$, and consequently $\|u_0\|_{W^{1,p}_r(B_1)}=1$. This together with \eqref{pb5} clearly contradicts the uniqueness (see Theorem~\ref{uniq}).

This contradiction proves our desired estimate.
\end{proof}

\subsection{The existence}
With the uniqueness theorem and the apriori estimate, we are now ready to give a existence theorem to the problem~\eqref{pb3} by using the continuity method.
\begin{theorem}
Let $\frac{1}{2}<s<1$, $1-s\leqslant r\leqslant s$, $1\leqslant p<\infty$, suppose $\vec b, c\in L^\infty(B_1)$, $c\geqslant 0$ in $B_1$ and $f\in  L^p_r(B_1)$. Then there exists a unique $u\in \mathscr{W}^{2s,p}_{r,0}(B_1)$
that solves the problem \eqref{pb3}
\end{theorem}
\begin{proof}

To prove this theorem, we use the method of continuity.

First, we define $\mathscr{L}_\tau:\mathscr{W}^{2s,p}_{r,0}(B_1)\rightarrow  L^p_r(B_1)$,  as:
$ \mathscr{L}_\tau u=(-\Delta)^s u+\tau(\vec b\cdot\nabla u+cu).$ and consider the following problems:

\begin{equation}\label{pb6}
\left\{
\begin{aligned}
\mathscr{L}_\tau u&=f\ \ \ &\text{in}\ &\mathcal{D}'(B_1),\\
u&=0\ \ \ &\text{in}\ &\mathbb{R}^n\backslash B_1.
\end{aligned}
\right.
\end{equation}
We say $\mathscr{L}_\tau:\mathscr{W}^{2s,p}_{r,0}(B_1)\rightarrow  L^p_r(B_1)$ is invertible, if for any given $f\in  L^p_r(B_1)$ there exists a unique $u\in \mathscr{W}^{2s,p}_{r,0}(B_1)$ that solves \eqref{pb6}.
Clearly, $\mathscr{L}_0=(-\Delta)^s:
\mathscr{W}^{2s,p}_{r,0}(B_1)\rightarrow  L^p_r(B_1)$ is invertible (see Corollary \ref{cor1}).
Hence, we can define the non-empty set:
\begin{equation}
E=\{\sigma\in[0,1]|\mathscr{L}_\sigma:\mathscr{W}^{2s,p}_{r,0}(B_1)\rightarrow  L^p_r(B_1)\ {\rm is\ an\ invertible\ operator}\}\ni 0.
\end{equation}

Suppose that $\sigma\in E$, write $\mathscr{L}_\sigma^{-1}: L^p_r(B_1)\rightarrow \mathscr{W}^{2s,p}_{r,0}(B_1)$ as the inverse of $\mathscr{L}_\sigma$.
For any $\tau\in [0,1]$ and $f\in  L^p_r(B_1)$, we know $u\in \mathscr{W}^{2s,p}_{r,0}(B_1)$ is a solution of \eqref{pb6} if and only if
\begin{equation}
\mathscr{L}_\sigma u=f+(\mathscr{L}_\sigma-\mathscr{L}_\tau)u=f+(\sigma-\tau)(\vec b\cdot\nabla+c)u\quad\text{in}\ \mathcal{D}'(B_1).
\end{equation}
i.e.
\begin{equation}
u=\mathscr{L}_\sigma^{-1}f+(\sigma-\tau)\mathscr{L}_\sigma^{-1}(\vec b\cdot\nabla+c)u
\end{equation}
Define an operator $T:\mathscr{W}^{2s,p}_{r,0}(B_1)\rightarrow \mathscr{W}^{2s,p}_{r,0}(B_1)$ as:
\begin{equation}
Tv=\mathscr{L}_\sigma^{-1}f+(\sigma-\tau)\mathscr{L}_\sigma^{-1}(\vec b\cdot\nabla+c)v.
\end{equation}
Then for any $v_1, v_2\in \mathscr{W}^{2s,p}_{r,0}(B_1)$, we know from Theorem \ref{apest4} that
\begin{equation}
\begin{aligned}
\big\|(-\Delta)^s(Tv_1-Tv_2)\big\|_{ L^p_r(B_1)}
&=|\sigma-\tau|\Big\|(-\Delta)^s\big[\mathscr{L}_\sigma^{-1}(\vec b\cdot\nabla+c)(v_1-v_2)\big]\Big\|_{ L^p_r(B_1)}\\
&\leqslant C_1|\sigma-\tau|\big\|(\vec b\cdot\nabla+c)(v_1-v_2)\big\|_{ L^p_r(B_1)}\\
&\leqslant C_1\Lambda|\sigma-\tau|\Big(\big\||\nabla(v_1-v_2)|\big\|
_{ L^p_r(B_1)}+\big\|v_1-v_2\big\|_{ L^p_r(B_1)}\Big)\\
&\leqslant C_2|\sigma-\tau|\|(-\Delta)^s(v_1-v_2)\|_{ L^p_r(B_1)},
\end{aligned}
\end{equation}

Hence, $T:\mathscr{W}^{2s,p}_{r,0}(B_1)\rightarrow \mathscr{W}^{2s,p}_{r,0}(B_1)$ is a contraction for $|\sigma-\tau|<\epsilon_0=C_2^{-1}$.

For any $\tau\in[0,1]$ with $|\sigma-\tau|<\epsilon_0$, the Banach fixed point theorem guarantees a unique $u\in \mathscr{W}^{2s,p}_{r,0}(B_1)$ such that $u=Tu$, i.e. $\mathscr{L}_\tau u=f$ in $\mathcal{D}'(B_1)$. Hence, $\tau\in E$ whenever $|\sigma-\tau|<\epsilon_0$ for some $\sigma\in E$.

Then, by dividing $[0,1]$ into $[\epsilon_0^{-1}]+1$ parts, we derive $[0,1]\subseteq E$. Now we have $1\in E$, i.e. for all $f\in  L^p_r(B_1)$ there exists a unique $u\in \mathscr{W}^{2s,p}_{r,0}(B_1)$ that solves \eqref{pb3}.
\end{proof}


\begin{thebibliography}{10}

\bibitem{abatangelo2018green}
{\sc N.~Abatangelo, S.~Jarohs, and A.~Salda{\~n}a}, {\em Green function and
  Martin kernel for higher-order fractional Laplacians in balls}, Nonlinear
  analysis, 175 (2018), pp.~173--190.

\bibitem{bucur}
{\sc C.~Bucur}, {\em Some observations on the Green function for the ball in
  the fractional Laplace framework}, COMMUNICATIONS ON PURE AND APPLIED
  ANALYSIS, 15 (2016), pp.~657--699.

\bibitem{cabre1}
{\sc X.~Cabr{\'e} and Y.~Sire}, {\em Nonlinear equations for fractional
 Laplacians, i: Regularity, maximum principles, and Hamiltonian estimates}, in
  Annales de l'Institut Henri Poincare (C) Non Linear Analysis, vol.~31,
  Elsevier, 2014, pp.~23--53.

\bibitem{caffarelli2}
{\sc L.~Caffarelli and L.~Silvestre}, {\em An extension problem related to the
  fractional Laplacian}, Communications in partial differential equations, 32
  (2007), pp.~1245--1260.


  \bibitem{2010Regularity}
{\sc L.~Caffarelli and L.~Silvestre}, {\em Regularity theory for fully
  nonlinear integro-differential equations}, Communications on Pure \& Applied
  Mathematics, 62 (2010).

\bibitem{2017Maximum}
{\sc W.~Chen and C.~Li}, {\em Maximum principles for the fractional p-Laplacian
  and symmetry of solutions},  (2017).

\bibitem{chen2}
{\sc W.~Chen, C.~Li, and Y.~Li}, {\em A direct method of moving planes for the
  fractional Laplacian}, Advances in Mathematics, 308 (2017), pp.~404--437.

  \bibitem{2006Classification}
{\sc W.~Chen, C.~Li, and B.~Ou}, {\em Classification of solutions for an
  integral equation}, Communications on Pure and Applied Mathematics, 59
  (2006), pp.~330--343.

\bibitem{D2016ON}
{\sc J.~Dávila, L.~Dupaigne, and J.~Wei}, {\em On the fractional Lane-Emden
  equation}, Transactions of the American Mathematical Society, 369 (2016),
  pp.~6087--6104.

\bibitem{2017On}
{\sc M.~Fazly and J.~Wei}, {\em On finite {M}orse index solutions of higher
  order fractional {L}ane-{E}mden equations}, American Journal of Mathematics,
  139 (2017), pp.~433--460.

\bibitem{Gilbarg1991Elliptic}
{\sc D.~Gilbarg and N.~S. Trudinger}, {\em Elliptic partial differential
  equations of second order}, Grundlehren Der Mathematischen Wissenschaften,
  224 (1991), pp.~469--484.

  \bibitem{2016Existence}
{\sc S.~Kim, M.~Musso, and J.~Wei}, {\em Existence theorems of the fractional
  Yamabe problem}, Analysis \& PDE, 11 (2016).

\bibitem{1996Local}
{\sc C.~Li}, {\em Local asymptotic symmetry of singular solutions to nonlinear
  elliptic equations}, Inventiones Mathematicae, 123 (1996), pp.~221--231.

\bibitem{li2020non}
{\sc C.~Li, C.~Liu, Z.~Wu, and H.~Xu}, {\em Non-negative solutions to
  fractional Laplace equations with isolated singularity}, Advances in
  Mathematics, 373 (2020), p.~107329.

\bibitem{li0}
{\sc C.~Li, Z.~Wu, and H.~Xu}, {\em Maximum principles and B{\^o}cher type
  theorems}, Proceedings of the National Academy of Sciences,  (2018),
  p.~201804225.

\bibitem{ros}
{\sc X.~Ros-Oton and J.~Serra}, {\em The Dirichlet problem for the fractional
 Laplacian: regularity up to the boundary}, Journal de Math{\'e}matiques Pures
  et Appliqu{\'e}es, 101 (2014), pp.~275--302.

\bibitem{silvestre}
{\sc L.~Silvestre}, {\em Regularity of the obstacle problem for a fractional
  power of the Laplace operator}, Communications on Pure and Applied
  Mathematics: A Journal Issued by the Courant Institute of Mathematical
  Sciences, 60 (2007), pp.~67--112.

\bibitem{2019Symmetry}
{\sc Z.~Wu and H.~Xu}, {\em Symmetry properties in systems of fractional
  {L}aplacian equations}, Discrete and Continuous Dynamical Systems, 39 (2019),
  pp.~1559--1571.

\bibitem{2015Non}
{\sc R.~Zhang}, {\em Non-local curvature and topology of locally conformally
  flat manifolds}, Advances in Mathematics, 335 (2015), pp.~587--613.
\end{thebibliography}
\end{document}